\documentclass[12pt,a4paper]{amsart}
\usepackage{amsmath}
\usepackage{amsthm}
\usepackage{amssymb}
\usepackage{mathrsfs}
\usepackage[english]{babel}
\usepackage{verbatim}
\usepackage[shortlabels]{enumitem}

\newtheoremstyle{mio}%
	{}{} 
	{\itshape}{} 
	{\bfseries}{.}{ } 
	{#1 #2\thmnote{\mdseries~(\scshape #3)}} 
\theoremstyle{mio}
\newtheorem{teor}{Theorem}[section]
\newtheorem{cor}[teor]{Corollary}
\newtheorem{prop}[teor]{Proposition}
\newtheorem{lemma}[teor]{Lemma}
\newtheorem{defin}[teor]{Definition}

\newtheoremstyle{definition2}%
	{}{} 
	{}{} 
	{\bfseries}{.}{ } 
	{#1 #2\thmnote{\mdseries~ #3}} 
\theoremstyle{definition2}
\newtheorem{oss}[teor]{Remark}
\newtheorem{ex}[teor]{Example}

\newcommand{\qspec}[1]{\mathrm{QSpec}^{#1}}
\newcommand{\psspec}[1]{\mathrm{PsSpec}^{#1}}
\newcommand{\norm}[1]{\widehat{#1}}

\newcommand{\inssubmod}{\mathbf{F}}

\newcommand{\inssemistar}{\mathrm{SStar}}
\newcommand{\inssmstar}{\mathrm{(S)Star}}
\newcommand{\Min}{\mathrm{Min}}
\newcommand{\inssemistable}{\inssemistar_{st}}
\newcommand{\inssmstable}{\inssmstar_{st}}

\newcommand{\inssemisv}{\inssemistar_{sv}}
\newcommand{\xxcal}{\mathscr{X}}
\newcommand{\power}{\mathcal{P}}

\newcommand{\insfracid}{\mathcal{F}}
\DeclareMathOperator{\Spec}{Spec}
\DeclareMathOperator{\Max}{Max}
\DeclareMathOperator{\rad}{rad}

\newcommand{\ins}[1]{\mathbb{#1}}

\newcommand{\insQ}{\ins{Q}}
\newcommand{\Zar}{\mathrm{Zar}}

\title[Stable semistar operations on a Pr\"ufer domain]{Towards a classification of stable semistar operations on a Pr\"ufer domain}
\author{Dario Spirito}
\address{Dipartimento di Matematica e Fisica, Universit\`a degli Studi ``Roma Tre'', Roma, Italy}
\email{spirito@mat.uniroma3.it}

\begin{document}
\begin{abstract}
We study stable semistar operations defined over a Pr\"ufer domain, showing that, if every ideal of a Pr\"ufer domain $R$ has only finitely many minimal primes, every such closure can be described through semistar operations defined on valuation overrings of $R$.
\end{abstract}

\maketitle

\section{Introduction}
Semistar operations were defined and studied by Okabe and Matsuda in \cite{okabe-matsuda} as a more flexible version of the classical notion of star operation, first introduced by Krull \cite{krull_breitage_I-II} and Gilmer \cite[Chapter 32]{gilmer}. Several distinguished classes of star and semistar operations have been investigated: among these, we can cite \emph{finite-type} operations, \emph{spectral} operations (linked to the spectrum of the ring; see e.g. \cite{anderson_two_2000,anderson_intersections_2005,localizing-semistar})  and \emph{eab} operations (linked with the valuation overrings of the ring; cfr., for example, \cite{fontana_loper-eab} and \cite[Section 4]{fifolo_transactions}).

The aim of the present paper is to study and classify \emph{stable} semistar operations, that is, semistar operations that distribute over finite intersections. This class of closure operations is closely linked with spectral operations and, indeed, the two concepts are often introduced together \cite{anderson_two_2000,anderson_intersections_2005}. However, while spectral operations, due to their definition, can be classified in a simple way by studying subsets of the spectrum of the ring (\cite[Remark 4.5]{localizing-semistar} and \cite[Corollary 4.4]{topological-cons}), stable operations require more work, and their classification isn't nearly as clear as the classification of spectral operations. We prove that, if $R$ is a Pr\"ufer domain such that every ideal has only finitely many minimal primes, then stable operations have a standard representation (Corollary \ref{cor:MinIfin-glob}), but that this result isn't general enough to cover all cases (Example \ref{ex:AD}); moreover, we show that, if $R$ is a Pr\"ufer domain with Noetherian spectrum, a stable operation $\star$ such that $R=R^\star$ is uniquely determined by a subset of the set $\mathcal{M}$ of non-divisorial maximal ideals, and that the set of this closures is order-isomorphic to the power set of $\mathcal{M}$ (Proposition \ref{prop:noeth-max}).

\section{Preliminaries}\label{sect:prel}
Let $R$ be an integral domain with quotient field $K$, let $\inssubmod(R)$ be the set of $R$-submodules of $K$ and let $\insfracid(R)$ be the set of fractional ideals of $R$ (where a $R$-submodule $I$ of $K$ is a \emph{fractional ideal} of $R$ if $dI\subseteq K$ for some $d\in K\setminus\{0\}$). A \emph{semistar operation} on $R$ is a map $\star:\inssubmod(R)\longrightarrow\inssubmod(R)$, $I\mapsto I^\star$, such that, for every $I,J\in\inssubmod(R)$, $x\in K$,
\begin{enumerate}
\item $I\subseteq I^\star$;
\item $I^\star=(I^\star)^\star$;
\item if $I\subseteq J$, then $I^\star\subseteq J^\star$;
\item $(xI)^\star=x\cdot I^\star$.
\end{enumerate}
If $R=R^\star$, then $\star$ is said to be a \emph{(semi)star operation}; if $\star$ is (semi)star, then $\star|_{\insfracid(R)}$ is said to be a \emph{star operation} (this is equivalent to the usual definition, which does not uses semistar operations; see \cite[Chapter 32]{gilmer}). A submodule $I$ such that $I=I^\star$ is said to be \emph{$\star$-closed}, while an integral ideal $I$ such that $I=I^\star\cap R$ is said to be \emph{quasi-$\star$-closed}. The set of $\star$-closed submodules, and the set of quasi-$\star$-closed ideals, are closed by arbitrary intersections.

We will be using the following property of star and semistar operation: if $I$ is $\star$-closed, then so is $(I:J):=\{x\in K\mid xJ\subseteq I\}$, for every submodule $J$ of the quotient field $K$.

The set of semistar (respectively, star) operations has an order, defined by $\star_1\leq\star_2$ if and only if $I^{\star_1}\subseteq I^{\star_2}$ for all submodules (resp., fractional ideals) $I$. Every family $\Delta$ of semistar (or (semi)star, or star) operations has an infimum, given by the map $I\mapsto\bigcap_{\star\in\Delta}I^\star$; it also has a supremum, and a submodule is closed by $\sup\Delta$ if and only it is $\star$-closed for every $\star\in\Delta$. Among star operations, the biggest is the \emph{$v$-operation}, defined as $v:I\mapsto(R:(R:I))$. Ideals that are closed by the $v$-operations are said to be \emph{divisorial}.

A semistar (or (semi)star, or star) operation is said to be \emph{stable} if $(I\cap J)^\star=I^\star\cap J^\star$ for every $I,J\in\inssubmod(R)$. If $\star$ is a (semi)star stable operation, then $\star|_{\insfracid(R)}$ is a stable star operation; moreover, there is a one-to-one correspondence between stable star operations and stable (semi)star operations (see \cite[Discussion after Proposition 3.10]{surveygraz} or \cite[Proposition 3.4]{spettrali-eab}). Two stable semistar operations $\star_1,\star_2$ on $R$ are equal if and only if, for every proper ideal $I$ of $R$, $I^{\star_1}=R^{\star_1}$ is equivalent to $I^{\star_2}=R^{\star_2}$ (see \cite[Theorem 2.6]{anderson_two_2000} and \cite[p.182]{localizing-semistar}), and thus if and only if, for every integral ideal $I$, $1\in I^{\star_1}$ is equivalent to $1\in I^{\star_2}$.

If $\star$ is a stable semistar operation of $R$, $I$ and $J$ are $R$-modules, and $J=j_1R+\cdots+j_nR$ is finitely generated, then $(I:J)^\star=(I^\star:J^\star)=(I^\star:J)$, because $(I:J)=j_1^{-1}I\cap\cdots\cap j_n^{-1}I$.

We denote by $\inssemistable(R)$ the set of stable semistar operations on $R$.

Let now $V$ be a valuation domain with quotient field $K$. Then, there are at most two star operations on $V$: the identity (denoted by $d$) and the $v$-operation. Moreover, $d\neq v$ if and only if the maximal ideal $M$ of $V$ is not principal, in which case the only fractional ideals $I$ such that $I\neq I^v$ are those of the form $xM$, as $x$ ranges in $K\setminus\{0\}$ \cite[Chapter 31, Exercise 12]{gilmer}. In particular, if $xM\subseteq I$, then $x\in I^v$.

Both the identity and the $v$-operation are stable since, if $I,J\in\insfracid(V)$, then (without loss of generality) $I\subseteq J$, and thus $(I\cap J)^v=I^v=I^v\cap J^v$.

\section{The two spectra}
Let $\star$ be a stable semistar operation on the domain $R$. The $R$-module $D:=R^\star$ is still a ring, and the closure of every $D$-module is again a $D$-module; it follows that $\star|_{\inssubmod(D)}$ is again a semistar operation, and clearly it is stable. It follows that there is no loss of generality in supposing that $R=R^\star$, i.e., that $\star$ is a (semi)star operation.

We start by studying the action of $\star$ on prime and primary ideals.
\begin{lemma}\label{lemma:spec}
Let $\star$ be a stable (semi)star operation on $R$, and let $P\in\Spec(R)$. Then, $P^\star\in\{P,R\}$.
\end{lemma}
\begin{proof}
Since $R^\star=R$, we have $P^\star\subseteq R$. Suppose $P\neq P^\star$; then, there is an $x\in P^\star\setminus P$, and in particular $R\subseteq(P^\star:x)$. Moreover, since $P$ is prime and $x\in R$, one has $(P:x)\cap R=P$. Therefore,
\begin{equation*}
P^\star=((P:x)\cap R)^\star=(P:x)^\star\cap R^\star=(P^\star:x)\cap R\ni 1,
\end{equation*}
and thus $R\subseteq P^\star$. Hence, $P^\star=R$.
\end{proof}

\begin{lemma}\label{lemma:prim}
Let $\star$ be a stable (semi)star operation on $R$, and let $L$ be a $P$-primary ideal of $R$.
\begin{enumerate}[(a)]
\item\label{lemma:prim:prim} $L^\star=(LR_P)^\star\cap R$.
\item\label{lemma:prim:pruf} If $R$ is a Pr\"ufer domain, then $L$ is $\star$-closed if and only if $LR_P$ is $\star$-closed.
\end{enumerate}
\end{lemma}
\begin{proof}
\ref{lemma:prim:prim} Since $L$ is $P$-primary, $L=LR_P\cap R$; therefore,
\begin{equation*}
L^\star=(LR_P\cap R)^\star=(LR_P)^\star\cap R^\star=(LR_P)^\star\cap R,
\end{equation*}
as claimed.

\ref{lemma:prim:pruf} If $LR_P$ is $\star$-closed, then, by the previous point, $L^\star=LR_P\cap R=L$.

Conversely, suppose $L=L^\star$. The set $(LR_P)^\star$ is a $R_P$-module; since $P$ is maximal in $R_P$, and $R_P$ is a valuation domain, it follows that either $1\in(LR_P)^\star$ or $LR_P$ is $PR_P$-primary. The former case is impossible, since $1\notin L=(LR_P)^\star\cap R$; however, if $(LR_P)^\star$ is $PR_P$-primary and different from $LR_P$, then $(LR_P)^\star\cap R$ cannot be equal to $L$, since there is a one-to-one correspondence between the $P$-primary ideals of $R$ and the $PR_P$-primary ideals of $R_P$. Therefore, $LR_P$ must be $\star$-closed.
\end{proof}

Let $\star$ be a semistar operation on the integral domain $R$. The \emph{quasi-spectrum} of $\star$, denoted by $\qspec{\star}(R)$, is the set of prime ideals $P$ of $R$ such that $P=P^\star\cap R$. If $\star$ is a stable semistar operation, and we suppose that $\star$ is \emph{of finite type}, i.e., that
\begin{equation*}
I^\star=\bigcup\{J^\star\mid J\subseteq I,~J\text{~is finitely generated}\},
\end{equation*}
then $\star$ is uniquely determined by the set $\qspec{\star}(R)$ (see \cite[Corollary 4.2]{anderson_overrings_1988} and \cite[page 185 and Theorem 4.12(3)]{localizing-semistar}). This does not happen if we drop the assumption that $\star$ is of finite type: for example, if $V$ is a one-dimensional valuation domain with nonprincipal maximal ideal $M$, then the quasi-spectrum of the $v$-operation and of the semistar operation $\wedge_{\{K\}}$ (that sends every nonzero $V$-submodule of $K$ to $K$) are both equal to $\{(0)\}$. However, the $v$-operation closes every $M$-primary ideal (except $M$ itself), while $\wedge_{\{K\}}$ does not. This suggests the following definition.
\begin{defin}
Let $\star$ be a stable semistar operation on the integral domain $R$. The \emph{pseudo-spectrum} $\psspec{\star}(R)$ of $\star$ is the set of all prime ideals $P$ such that $P^\star\cap R=R$ (i.e., $P^\star=R^\star$) and $L^\star\cap R=L$ for some $P$-primary ideal $L$ of $R$.
\end{defin}

\begin{prop}\label{prop:qspec-psspec}
Let $R$ be an integral domain, let $P\subsetneq Q$ be prime ideals of $R$, and let $\star$ be a stable (semi)star operation on $R$.
\begin{enumerate}[(a)]
\item If $Q\in\qspec{\star}(R)$, then $P\in\qspec{\star}(R)$.
\item If $R$ is Pr\"ufer and $Q\in\psspec{\star}(R)$, then $P\in\qspec{\star}(R)$; in particular, no two members of $\psspec{\star}(R)$ are comparable.
\end{enumerate}
\end{prop}
\begin{proof}
If $Q\in\qspec{\star}(R)$, then $P^\star\subseteq Q^\star=Q$; by Lemma \ref{lemma:prim}\ref{lemma:prim:prim}, $P^\star$ must be equal to $P$, and thus $P\in\qspec{\star}(R)$.

Suppose now that $R$ is Pr\"ufer and that $Q\in\psspec{\star}(R)$. By hypothesis, there is a $Q$-primary ideal $L$ such that $L=L^\star$.

We claim that $PR_Q=\bigcap\{xLR_Q:x\in Q\setminus P\}$. Note first that $xLR_Q$ is contained in $QR_Q$ for each $x\in Q\setminus P$. If $x\notin P$, then $xLR_Q\nsubseteq PR_Q$ and thus (being $R_Q$ a valuation domain) $PR_Q\subseteq xLR_Q$: thus $PR_Q$ is in the intersection. Conversely, if $y\in Q\setminus P$, then $y\notin yLR_Q$ (since $1\notin LR_Q\subseteq QR_Q$), and thus $y$ is not in the intersection.

However, since $L$ is $\star$-closed, so is $LR_Q$ (Lemma \ref{lemma:prim}\ref{lemma:prim:pruf}); hence, every $xLR_Q$ is $\star$-closed, and since the intersection of $\star$-closed ideals is $\star$-closed, so is $PR_Q$. Thus, $PR_Q\cap R=P$ is $\star$-closed, and $P\in\qspec{\star}(R)$.

The last claim follows directly from the previous part and the fact that $\qspec{\star}(R)$ and $\psspec{\star}(R)$ are disjoint by definition.
\end{proof}

While every prime ideal can be a in the quasi-spectrum of some stable operation $\star$ (for example, when $\star$ is the identity), the same does not happen for the pseudo-spectrum.
\begin{lemma}\label{lemma:princ-ps}
Let $R$ be a Pr\"ufer domain, and let $\star$ be a stable (semi)star operation on $R$. If $PR_P$ is principal over $R_P$, then $P\notin\psspec{\star}(R)$.
\end{lemma}
\begin{proof}
Suppose there is a $P$-primary ideal $L$ that is $\star$-closed. By Lemma \ref{lemma:prim}\ref{lemma:prim:pruf}, $LR_P$ is $\star$-closed; however, $LR_P$ is $PR_P$-primary, and since $R_P$ is a valuation domain and $PR_P=pR_P$ is principal we have $LR_P=aR_P$ for some element $a$. Therefore, $PR_P=pa^{-1}LR_P$ is $\star$-closed; applying again Lemma \ref{lemma:prim}\ref{lemma:prim:pruf}, we have that $P$ is $\star$-closed, and thus $P\in\qspec{\star}(R)$.
\end{proof}

Recall that a prime ideal $P$ of a domain $R$ is \emph{branched} if there exist a $P$-primary ideal different from $P$ (see e.g. \cite[Chapter 17]{gilmer}).
\begin{lemma}\label{lemma:zarclosed}
Let $R$ be a Pr\"ufer domain, $\star$ a stable (semi)star operation on $R$, and let $P$ be a prime ideal of $R$. Then:
\begin{enumerate}[(a)]
\item\label{lemma:zarclosed:a} if $P\in\qspec{\star}(R)\cup\psspec{\star}(R)$, then $R_P$ is $\star$-closed;
\item\label{lemma:zarclosed:b} if $P$ is branched and $R_P$ is $\star$-closed, then $P\in\qspec{\star}(R)\cup\psspec{\star}(R)$.
\end{enumerate}
\end{lemma}
\begin{proof}
If $P\in\qspec{\star}(R)\cup\psspec{\star}(R)$, there is a $P$-primary ideal $L$ (possibly equal to $P$) that is $\star$-closed. By Lemma \ref{lemma:prim}, it follows that $LR_P$ is $\star$-closed, and thus $(LR_P:LR_P)$ is $\star$-closed. We claim that $(LR_P:LR_P)=R_P$. Indeed, clearly $R_P\subseteq(LR_P:LR_P)$, and if the containment is strict then $LV\subseteq LR_P$ for some proper overring $V$ of $R_P$. However, since $LR_P$ is $PR_P$-primary and $R_P$ is a valuation domain, it follows that $LV=V$, and the inclusion $LV\subseteq LR_P$ would imply $1\in LR_P$, a contradiction. Hence, $R_P$ is $\star$-closed.

Conversely, suppose that $R_P$ is $\star$-closed and that $P$ is branched. The latter property implies that there is an element $x\in R_P$ such that $PR_P$ is minimal over $xR_P$; hence, $xR_P$ is $PR_P$-primary and $\star$-closed. Thus, $xR_P\cap R$ is a $P$-primary $\star$-closed ideal, and thus $P$ is either in $\qspec{\star}(R)$ or in $\psspec{\star}(R)$.
\end{proof}

Our aim is to study how much, if $R$ is a Pr\"ufer domain, the behaviour of $\star$ on $\inssubmod(R_P)$ is determined on whether $P$ is contained in $\qspec{\star}(R)$, in $\psspec{\star}(R)$ or in neither. 
\begin{prop}\label{prop:leq}
Let $\star$ be a stable (semi)star operation on $R$, let $P\in\Spec(R)$ and let $I$ be a fractional ideal of $R$. Then:
\begin{enumerate}[(a)]
\item\label{prop:leq:ps} if $P\in\psspec{\star}(R)$, then $I^\star\subseteq (IR_P)^{v_{R_P}}$, where $v_{R_P}$ is the $v$-operation on $R_P$;
\item\label{prop:leq:qs} if $P\in\qspec{\star}(R)$, then $I^\star\subseteq IR_P$.
\end{enumerate}
\end{prop}
\begin{proof}
Without loss of generality, we can suppose that $I\subseteq R$. By Lemma \ref{lemma:zarclosed}, under both $P\in\qspec{\star}(R)$ and $P\in\psspec{\star}(R)$ the overring $R_P$ is $\star$-closed, and thus $\star|_{\inssubmod(R_P)}$ is a (semi)star operation on $R_P$: in particular, $\star|_{\inssubmod(R_P)}$ must be equal to $d_{R_P}$ (i.e., the identity on $R_P$) or $v_{R_P}$.

In both cases, $\star|_{\inssubmod(R_P)}\leq v_{R_P}$, and thus $I^\star\subseteq(IR_P)^{v_{R_P}}$, proving \ref{prop:leq:ps}. If, moreover, $P\in\qspec{\star}(R)$, then by Lemma \ref{lemma:prim} also $PR_P$ is closed, and thus $\star|_{\inssubmod(R_P)}$ must be $d_{R_P}$, and $I^\star\subseteq IR_P$. \ref{prop:leq:qs} is proved.
\end{proof}

\begin{lemma}\label{lemma:est-spettri}
Let $R$ be a Pr\"ufer domain and $D$ an overring of $R$; let $\star\in\inssemistable(R)$ and let $\sharp:=\star|_{\inssubmod(D)}\in\inssemistable(D)$. Then,
\begin{enumerate}[(a)]
\item $\qspec{\sharp}(D):=\{PD\mid P\in\qspec{\star}(R),PD\neq D\}$;
\item $\psspec{\sharp}(D):=\{PD\mid P\in\psspec{\star}(R),PD\neq D\}$.
\end{enumerate}
\end{lemma}
\begin{proof}
Note first that, since $R$ is a Pr\"ufer domain, the prime ideals of the overring $D$ are the extensions of the prime ideals $P$ of $R$ such that $PD\neq D$ and, for such ideals, $R_P=D_{PD}$ \cite[Theorem 26.1]{gilmer}.

Let $P\in\Spec(R)$ be such that $PD\neq D$. Both the $P$-primary ideals of $R$ and the $PD$-primary ideals of $D$ are in bijective correspondence with the $PR_P$-primary ideals of $R_P$ and, by Lemma \ref{lemma:prim}, such correspondence preserves whether the ideals are quasi-$\star$-closed (equivalently, quasi-$\sharp$-closed). The claim follows.
\end{proof}

We can now prove the main result of this section.
\begin{teor}\label{teor:bigleq}
Let $R$ be a Pr\"ufer domain and let $\star$ be a stable semistar operation on $R$. Then, for every $I\in\inssubmod(R)$,
\begin{equation}\label{eq:bigleq}
I^\star\subseteq\bigcap_{P\in\qspec{\star}(R)}IR_P\cap\bigcap_{P\in\psspec{\star}(R)}(IR_P)^{v_{R_P}},
\end{equation}
where $v_{R_P}$ is the $v$-operation on $R_P$.
\end{teor}
\begin{proof}
Since $R$ is Pr\"ufer, its overring $D:=R^\star$ is again a Pr\"ufer domain, and $I^\star=(ID)^\star$. By Lemma \ref{lemma:est-spettri}, and since $R_P=D_{PD}$ if $PD\neq D$, we can thus suppose $D=R$.

Let $x\in I^\star$; then, $1\in x^{-1}I^\star\cap R=(x^{-1}I\cap R)^\star$. However, $x^{-1}I\cap R$ is a fractional ideal of $R$; therefore, by Proposition \ref{prop:leq}, $1\in(x^{-1}I\cap R)R_P$ if $P\in\qspec{\star}(R)$, while $1\in[(x^{-1}I\cap R)R_P]^{v_{R_P}}$ if $P\in\psspec{\star}(R)$. In the former case, we have $1\in x^{-1}IR_P\cap R_P$, and thus $x\in IR_P$; in the latter, $1\in (x^{-1}IR_P)^{v_{R_P}}\cap R_P$, and thus $x\in (IR_P)^{v_{R_P}}$. The claim follows.
\end{proof}

\section{Classifying stable operations}
In the statement of Theorem \ref{teor:bigleq}, the right hand side of \eqref{eq:bigleq} is itself a semistar operation. Therefore, it is worthwhile to abstract it: given a stable semistar operation $\star$ on the Pr\"ufer domain $R$, we define the \emph{normalized stable version} of $\star$ as the semistar operation $\norm{\star}$ such that, for all $I\in\inssubmod(R)$,
\begin{equation*}
I^{\norm{\star}}:=\bigcap_{P\in\qspec{\star}(R)}IR_P\cap\bigcap_{P\in\psspec{\star}(R)}(IR_P)^{v_{R_P}}.
\end{equation*}

We collect the main properties of $\norm{\star}$ in the following proposition.
\begin{prop}\label{prop:normstar}
Let $R$ be a Pr\"ufer domain and let $\star$ be a stable semistar operation on $R$; let $\norm{\star}$ be its normalized stable version. Then:
\begin{enumerate}[(a)]
\item\label{prop:normstar:st} $\norm{\star}$ is a stable semistar operation;
\item\label{prop:normstar:leq} $\star\leq\norm{\star}$;
\item\label{prop:normstar:1P} if $P\in\Spec(R)$, then $1\in P^\star$ if and only if $1\in P^{\norm{\star}}$;
\item\label{prop:normstar:qs} $\qspec{\norm{\star}}(R)=\qspec{\star}(R)$;
\item\label{prop:normstar:ps} $\psspec{\norm{\star}}(R)=\psspec{\star}(R)$;
\item\label{prop:normstar:nn} $\norm{(\norm{\star})}=\norm{\star}$.
\end{enumerate}
\end{prop}
\begin{proof}
\ref{prop:normstar:st} Since the infimum of a family of stable operations is stable, we need to show that the semistar operations $d_P:I\mapsto IR_P$ and $v_P:I\mapsto(IR_P)^{v_{R_P}}$ are stable. By the flatness of $R_P$ over $R$, 
\begin{equation*}
(I\cap J)^{d_P}=(I\cap J)R_P=IR_P\cap JR_P=I^{d_P}\cap J^{d_P},
\end{equation*}
and thus $d_P$ is stable. Analogously, using the results at the end of Section \ref{sect:prel},
\begin{equation*}
\begin{array}{rcl}
(I\cap J)^{v_P} & = & ((I\cap J)R_P)^{v_{R_P}}=(IR_P\cap JR_P)^{v_{R_P}}=\\
& = & (IR_P)^{v_{R_P}}\cap(JR_P)^{v_{R_P}}=I^{v_P}\cap J^{v_P},
\end{array}
\end{equation*}
and $v_P$ is stable.

\ref{prop:normstar:leq} is exactly Theorem \ref{teor:bigleq}. 

From now on, let $P$ be a prime ideal of $R$.

\ref{prop:normstar:1P} If $1\in P^\star$ then $1\in P^{\norm{\star}}$ by Theorem \ref{teor:bigleq}; suppose $1\in P^{\norm{\star}}$. Then, $P\notin\qspec{\star}(R)$, and thus, by Lemma \ref{lemma:spec}, $1\in P^\star$.

\ref{prop:normstar:qs} If $P\in\qspec{\star}(R)$ then $P^{\norm{\star}}\cap R\subseteq PR_P\cap R=P$, and thus $P\in\qspec{\norm{\star}}(R)$. Conversely, suppose $P\in\qspec{\norm{\star}}(R)$. By the definition of $\norm{\star}$, this implies that $P\subseteq Q$ for some $Q\in\qspec{\star}(R)$ or that $P\subsetneq Q$ for some $Q\in\psspec{\star}(R)$. In both cases, $P\in\qspec{\star}(R)$ by Proposition \ref{prop:qspec-psspec}.

\ref{prop:normstar:ps} By \ref{prop:normstar:1P}, $P\notin\qspec{\star}(R)$ if and only if $P\notin\qspec{\norm{\star}}(R)$; therefore, we must prove that, if $L$ is a $P$-primary ideal, then it is quasi-$\star$-closed if and only if it is quasi-$\norm{\star}$-closed.

Suppose $P\in\psspec{\star}(R)$; then, $L=L^\star\cap R$ for some $P$-primary ideal $L$. By Lemma \ref{lemma:prim}\ref{lemma:prim:prim},
\begin{equation*}
L^{\norm{\star}}\cap R\subseteq (LR_P)^\star\cap R=L^\star\cap R=L,
\end{equation*}
and thus $P\in\psspec{\norm{\star}}(R)$. Conversely, if $P\in\psspec{\norm{\star}}(R)$, then $L^{\norm{\star}}\cap R=L$ for some $P$-primary ideal $L$. In particular, $L$ must be contained in some prime ideal $Q\in\qspec{\star}(R)\cup\psspec{\star}(R)$, and also $P$ must be contained in $Q$. If $Q\in\qspec{\star}(R)$ then $P\in\qspec{\norm{\star}}(R)=\qspec{\star}(R)$, a contradiction; if $Q\in\psspec{\star}(R)$, then $Q=P$, since otherwise $Q$ and $P$ would be comparable ideals in $\psspec{\star}(R)$, contradicting Lemma \ref{lemma:prim}\ref{lemma:prim:pruf}. Thus, $P\in\psspec{\star}(R)$.

\ref{prop:normstar:nn} follows by the previous two points.
\end{proof}

Our wish is that $\norm{\star}$ is actually equal to $\star$; however, this is not true in general, as the next example shows. 
\begin{ex}\label{ex:AD}
Let $R:=\ins{A}$ be the ring of all algebraic integers, and, for every $P\in\Max(R)$, let $\star_P$ be the semistar operation defined by
\begin{equation*}
I^{\star_P}:=\bigcap_{\substack{M\in\Max(R)\\ M\neq P}}IR_M.
\end{equation*}
By \cite[Example 4.5]{spettrali-eab}, each $\star_P$ is a (semi)star operation, and thus also $\star:=\sup\{\star_P\mid P\in\Max(R)\}$ is a (semi)star operation, and it is such that, for every $P$-primary ideal $L$, $L^\star=R$ (since $L^{\star_P}=R$).

Consider the stable operation $\overline{\star}$, defined by \cite[Definition 2.2]{anderson_two_2000}
\begin{equation*}
I^{\overline{\star}}:=\bigcup\{(I:E)\mid E\in\inssubmod(R),E^\star=R^\star\}.
\end{equation*}
Then, $1\in I^\star$ if and only if $1\in I^{\overline{\star}}$; hence, $L^{\overline{\star}}=R$ for every primary ideal $L$. This means that $\qspec{\star}(R)=\{(0)\}$ and $\psspec{\star}(R)=\emptyset$; it follows that $\norm{\overline{\star}}$ is nothing but the trivial extension $I\mapsto K$ (where $K=\overline{\insQ}$ is the quotient field of $R$). However, $\overline{\star}$ closes $R$ (this follows, for example, from the fact that $\overline{\star}|_{\insfracid(R)}$ is a star operation \cite[Theorem 2.4]{anderson_two_2000}); hence, $\overline{\star}\neq\norm{\overline{\star}}$.
\end{ex}

Thus, to obtain good results about $\norm{\star}$, we need to restrict either the class of ideals or the class of domains we consider.
\begin{prop}\label{prop:ugprim}
Let $R$ be a Pr\"ufer domain and let $\star$ be a stable (semi)star operation on $R$. If $L$ is a primary ideal of $R$, then $L^\star=L^{\norm{\star}}\cap R$.
\end{prop}
\begin{proof}
Let $P$ be the radical of $L$. We distinguish three cases.
\begin{itemize}[leftmargin=*]
\item $P\in\qspec{\star}(R)$. Then, $L^{\norm{\star}}\cap R\subseteq LR_P\cap R=L$, and thus $L=L^{\norm{\star}}\cap R=L^\star$.

\item $P\in\psspec{\star}(R)$. Then, by Proposition \ref{prop:qspec-psspec},
\begin{equation*}
L^{\norm{\star}}\cap R=(LR_P)^{v_{R_P}}\cap R.
\end{equation*}
If $LR_P=PR_P$, then $L$ must be exactly $P$, and thus $L^\star=R=L^{\norm{\star}}\cap R$. Analogously, if $LR_P=xPR_P$ for some $x$, then by Lemmas \ref{lemma:prim}\ref{lemma:prim:prim} and \ref{lemma:zarclosed}\ref{lemma:zarclosed:a}
\begin{equation*}
L^\star=(LR_P)^\star\cap R=(xPR_P)^\star\cap R=x(PR_P)^\star\cap R=xR_P\cap R=L^{\norm{\star}}\cap R,
\end{equation*}
with the last equality coming from the definition of $\norm{\star}$. Suppose $LR_P\neq yPR_P$ for all $y$. Then, $LR_P$ is divisorial in $R_P$, and thus $L^{\norm{\star}}\cap R=L$. Since $L^\star\subseteq L^{\norm{\star}}\cap R$ always, the two ideals must coincide.

\item $P\notin \qspec{\star}(R)\cup\psspec{\star}(R)$. If $P=L$ then the result follows by Proposition \ref{prop:normstar}\ref{prop:normstar:1P}. If $P\neq L$, then $P$ is branched; by Lemma \ref{lemma:zarclosed}, $R_P$ is not $\star$-closed, and thus no $PR_P$-primary $I$ ideal can be $\star$-closed (by the proof of Lemma \ref{lemma:zarclosed}), and $1\in I^\star$ for any such $I$. In particular, $1\in L^\star$ and $L^\star=R$.

\end{itemize}

In all cases, $L^\star=L^{\norm{\star}}\cap R$. The claim is proved.
\end{proof}

The previous proposition shows that $\norm{\star}$ behaves well on primary ideals. In view of the definition of stability, the previous result can also be extended to ideals that are finite intersections of primary ideals, i.e., ideals that have a primary decomposition. In the next proposition, we prove a slightly weaker property for a slightly larger class of ideals.

We start with a lemma.
\begin{lemma}\label{lemma:min-prim}
Let $R$ be a Pr\"ufer domain, and let $I$ be a proper ideal of $R$ whose radical is equal to $P\in\Spec(R)$.
\begin{enumerate}[(a)]
\item\label{lemma:min-prim:cont} If $L$ is an ideal with radical $P$ and $LR_P\subsetneq IR_P$, then $L\subseteq I$.
\item\label{lemma:min-prim:prim} $I$ contains a a $P$-primary ideal.
\end{enumerate}
\end{lemma}
\begin{proof}
\ref{lemma:min-prim:cont} Suppose $L\nsubseteq I$. Then, there is a maximal ideal $M$ of $R$ such that $LR_M\nsubseteq IR_M$; since $R_M$ is a valuation domain, this implies that $IR_M\subseteq LR_M$, and thus $IR_MR_P\subseteq LR_MR_P$. Since $\rad(I)=\rad(L)=P$, $M$ contains $P$, and thus $R_MR_P=R_P$; therefore, $IR_P\subseteq LR_P$, against the hypothesis $LR_P\subsetneq IR_P$. Hence, $L\subseteq I$.

\ref{lemma:min-prim:prim} If $P$ is not branched, then $P=I$ \cite[Theorem 23.3(e)]{gilmer}, and $P$ is the requested primary ideal. If $P$ is branched, it is minimal over a principal ideal $xR$, and thus $PR_P$ is the radical of $xR_P$. Moreover, $\rad(IR_P)=PR_P$, and thus $x^n\in IR_P$ for some integer $n$; let $L:=x^{n+1}R_P\cap R$. We claim that $L$ is the requested ideal. Indeed, it is $P$-primary since it is the restriction of a $PR_P$-primary ideal (since $\rad(x^{n+1}R_P)=\rad(xR_P)=PR_P$ and $PR_P$ is maximal in $R_P$), and $LR_P=x^{n+1}R_P\subsetneq IR_P$ since $x^{n+1}R_P\subsetneq x^nR_P\subseteq IR_P$. Hence, we can apply the previous point. 
\end{proof}

If $I$ is an ideal of $R$, we denote by $V(I)$ the set of prime ideals of $R$ containing $I$, and by $\Min(I)$ the set of its minimal primes.
\begin{teor}\label{teor:MinIfin}
Let $R$ be a Pr\"ufer domain, $\star$ a stable semistar operation on $R$, and let $I$ be a proper ideal of $R$ such that $\Min(I)$ is finite. Then, $1\in I^\star$ if and only if $1\in I^{\norm{\star}}$.
\end{teor}
\begin{proof}
If $1\in I^\star$, then $1\in I^{\norm{\star}}$ by Theorem \ref{teor:bigleq}.

Suppose $1\in I^{\norm{\star}}$, and let $\Min(I):=\{P_1,\ldots,P_n\}$. For each $i=1,\ldots,n$, let $T_i:=\bigcap\{R_Q\mid Q\in V(P)\}$; then, each maximal ideal containing $I$ survives in some $T_i$, and thus $I=IT_1\cap\cdots\cap IT_n\cap R$. Hence, $I^\star=(IT_1)^\star\cap\cdots\cap (IT_n)^\star\cap R^\star$, and analogously for $\norm{\star}$; hence, it is enough to show that $1\in(IT_i)^{\star}$ for every $i$.

Fix an $i$, and let $P:=P_i$ and $T:=T_i$. Since $V(P)$ is compact, the map $\flat:I\mapsto\bigcap\{IR_M\mid M\in V(P)\}$ is a finite-type semistar operation on $R$ \cite[Corollary 4.4]{topological-cons} and thus also on $T$; moreover, since $T$ is Pr\"ufer (being an overring of a Pr\"ufer domain) and closed by $\flat$, the map $\flat|_{\inssubmod(T_i)}$ must be the identity. In particular, the prime ideals of $T$ are exactly the extensions of the prime ideals of $R$ contained in some $M\in\Max(R)\cap V(P)$; therefore, $Q:=PT$ is contained in every maximal ideal of $T$, and (being $\Spec(T)$ a tree) any prime ideal of $T$ is comparable with $Q$.

Let now $J:=IT$ and $\sharp:=\star|_{\inssubmod(T)}$; then, $\Min(J)=\{Q\}$ (i.e., $\rad(J)=Q$). Since $\qspec{\sharp}(T)$ and $\psspec{\sharp}(T)$ are, respectively, the extensions of the prime ideals in $\qspec{\star}(R)$ and $\psspec{\star}(R)$ that survive in $T$ (Lemma \ref{lemma:est-spettri}), we have $\norm{\sharp}=\norm{\star}|_{\inssubmod(T)}$.

Since $1\in J^{\norm{\sharp}}$, we have $Q\notin\qspec{\sharp}(T)$. On the other hand, if $Q\notin\qspec{\sharp}(T)\cup\psspec{\sharp}(T)$, $1\in L^\sharp$ for every $Q$-primary ideal $L$. By Lemma \ref{lemma:min-prim}\ref{lemma:min-prim:prim} we can find a $L\subseteq J$, and thus $1\in J^\sharp$.

Suppose now $Q\in\psspec{\sharp}(T)$. Then $J^{\norm{\sharp}}\subseteq(JR_Q)^{v_{R_Q}}$, and thus (since $1\in J^{\norm{\sharp}}$) $JR_Q=QR_Q$. For every $q\in QR_Q$, we have $qQR_Q\subsetneq JR_Q$; in particular, if $q\in Q$, by Lemma \ref{lemma:min-prim}\ref{lemma:min-prim:cont} we have $qQ\subseteq J$. Therefore,
\begin{equation*}
q=q\cdot 1\in qQ^\sharp=(qQ)^\sharp\subseteq J^\sharp,
\end{equation*}
and so $Q\subseteq I^\sharp$. Hence, $Q^\sharp\subseteq J^\sharp$, and so $1\in J^\sharp$.

Therefore, $1\in J^\sharp$ in every case, as requested.
\end{proof}

In a global perspective, we get immediately the following result. 
\begin{cor}\label{cor:MinIfin-glob}
Let $R$ be a Pr\"ufer domain such that every proper ideal has only a finite number of minimal primes. Then, $\star=\norm{\star}$ for every stable semistar operation $\star$ on $R$.
\end{cor}
\begin{proof}
Since $\star$ and $\norm{\star}$ are both stable, it suffices to show that the set of (proper) ideals $I$ of $R$ such that $1\in I^\star$ coincide with the set of ideals such that $1\in I^{\norm{\star}}$. However, this follows from Theorem \ref{teor:MinIfin}.
\end{proof}

\begin{cor}\label{cor:semiloc}
Let $R$ be a semilocal Pr\"ufer domain. Then, $\star=\norm{\star}$ for every stable semistar operation $\star$ on $R$.
\end{cor}
\begin{proof}
Let $I$ be a proper ideal of $R$; then, $V(I)\cap\Max(R)$ is finite. Moreover, every $M\in V(I)\cap\Max(R)$ contains only one minimal prime of $I$, since $\Spec(R)$ is a tree. The claim follows from Corollary \ref{cor:MinIfin-glob}.
\end{proof}

Recall that a topological space is \emph{Noetherian} if every subset if compact, or equivalently if every ascending chain of radical ideals stabilizes.
\begin{cor}\label{cor:noethspec}
Let $R$ be a Pr\"ufer domain with Noetherian spectrum. Then, $\star=\norm{\star}$ for every stable semistar operation $\star$ on $R$.
\end{cor}
\begin{proof}
If $\Spec(R)$ is Noetherian, then every ideal has only finitely many minimal primes \cite[Chapter 6, Exercises 5 and 7]{atiyah}. We can apply Corollary \ref{cor:MinIfin-glob}.
\end{proof}

Hence, classifying the stable semistar operations on the Pr\"ufer domains considered above amounts to characterize the different $\norm{\star}$. 

We denote by $\Zar(R)$ the set of valuation overrings of $R$, and by $\inssemisv(R)$ the set of semistar operations $\star$ such that $R^\star$ is a valuation domain; if $X\subseteq\inssemisv(R)$, let $X^\uparrow:=\{\star\in\inssemisv(R)\mid\star\geq\star_1\text{~for some~}\star_1\in X\}$. Denote also by $\xxcal(R)$ the set of subsets $X\subseteq\inssemisv(R)$ such that $X=X^\uparrow$. 

From now on, with a slight abuse of notation, given a $V\in\Zar(R)$, we denote by $d_V$ both the identity star operation on $V$ and the semistar operation (on $R$) $I\mapsto IV$, and by $v_V$ both the $v$-operation on $V$ and the semistar operation (on $R$) defined by $I\mapsto(IV)^{v_V}$.

There is a natural map $\pi:\inssemisv(R)\longrightarrow\Zar(R)$, $\star\mapsto R^\star$: by the results recalled at the end of Section \ref{sect:prel}, for any $V\in\Zar(R)$, the fiber $\pi^{-1}(V)$ contains exactly $d_V$ and $v_V$, and thus it is either a singleton (when the maximal ideal of $V$ is principal) or it is composed of two elements.

\begin{lemma}\label{lemma:principal}
Let $R$ be an integral domain, and let $V$ be a valuation overring of $R$ with maximal ideal $M$; let $d_V$ and $v_V$ as above. Let $\star$ be a semistar operation on $R$. Then:
\begin{enumerate}[(a)]
\item $V$ is $\star$-closed if and only if $\star\leq v_V$;
\item $M$ is $\star$-closed if and only if $\star\leq d_V$.
\end{enumerate}
\end{lemma}
\begin{proof}
If $\star\leq v_V$, then $V^\star\subseteq V^{v_V}=V$ and $V=V^\star$; similarly, if $\star\leq d_V$ then $M^\star\subseteq M^{d_V}=M$, so that $M=M^\star$.

Conversely, using the same proof of \cite[Lemma 3.1]{hhp_m-canonical} we see that, for any $L\in\inssubmod(R)$, the biggest semistar operation $\star$ such that $L$ is $\star$-closed is the map $I\mapsto(L:(L:I))$.

If $L=V$, then for any $I\in\inssubmod(R)$ we have
\begin{equation*}
(V:(V:I))=(V:(V:IV))=(IV)^{v_V};
\end{equation*}
hence $\star\leq v_V$. On the other hand, if $M$ is $\star$-closed then so is $(M:M)=V$; hence, $I\mapsto(M:(M:I))$ is a semistar operation that closes every $V$-submodule of the quotient field of $R$, and thus it must be $d_V$. Hence, $\star\leq d_V$.
\end{proof}

\begin{prop}\label{prop:xcal}
Let $R$ be a Pr\"ufer domain, and let $\Psi$ be the map
\begin{equation*}
\begin{aligned}
\Psi\colon\xxcal(R) & \longrightarrow \inssemistable(R)\\
X & \longmapsto \inf X.
\end{aligned}
\end{equation*}
Endow $\xxcal(R)$ with the reverse inclusion (i.e., $X\leq Y$ if $X\supseteq Y$). Then:
\begin{enumerate}[(a)]
\item\label{prop:xcal:wd} $\Psi$ is well-defined and order-preserving;
\item\label{prop:xcal:surj} if $\Min(I)$ is finite for every proper ideal $I$ of $R$, then $\Psi$ is surjective;
\item\label{prop:xcal:bij} if $\Spec(R)$ is Noetherian, then $\Psi$ is an order isomorphism.
\end{enumerate}
\end{prop}
\begin{proof}
\ref{prop:xcal:wd} With the same proof of Proposition \ref{prop:normstar}\ref{prop:normstar:st}, and since every $V\in\Zar(R)$ is in the form $R_P$ for some $P\in\Spec(R)$, we see that every $\star\in\xxcal(R)$ is stable, and thus $\Psi$ is well-defined. Moreover, if $X\supseteq Y$ then clearly $\inf X\leq\inf Y$, and thus $\Psi$ is order-preserving.

\ref{prop:xcal:surj} Given a stable semistar operation $\star$, let $X:=\{d_{R_P}\mid P\in\qspec{\star}(R)\}\cup\{v_{R_P}\mid P\in\psspec{\star}(R)\}$. By definition, $\inf X=\norm{\star}$; by Corollary \ref{cor:MinIfin-glob}, $\norm{\star}=\star$. Moreover, $\inf X=\inf X^\uparrow$, and thus $\star=\Psi(X^\uparrow)$.

\ref{prop:xcal:bij} We show that, for every $X,Y\in\xxcal(R)$, $X\neq Y$, if $\inf X\geq \inf Y$ then $X\subseteq Y$. Suppose not: then, there is a $\star\in X\setminus Y$. Let $V:=R^\star$, and let $M$ be the maximal ideal of $V$.

The map $\star|_{\insfracid(V)}$ is a star operation on $V$, and thus $\star$ is either equal to $d_V$ or to $v_V$. In the former case, $M^\star=M$, and so $M^{\inf X}=M$; moreover, for any semistar operation $\sharp$, $M^\sharp$ is a $V$-module, and thus if $M^\sharp\neq M$ then $V\subseteq M^\sharp$. Therefore, since $\inf X\geq\inf Y$, there must be a $\flat\in Y$ such that $M=M^\flat$. By Lemma \ref{lemma:principal}, $\flat\leq d_V=\star$. This contradicts the hypotheses $\star\notin Y$ and $Y=Y^\uparrow$, and thus this case is impossible.

Suppose $\star=v_V$. Since $\Spec(R)$ is Noetherian, $M$ is branched; in particular, $V$ has a smallest proper overring, say $W$. As in the previous case, $V^\star=V^{\inf X}=V$, and $V^\sharp$ is an overring of $V$ for every semistar operation $\sharp$ on $R$. Hence, if $V^\sharp\neq V$ then $W\subseteq V^\sharp$; thus, $V^\flat=V$ for some $\flat\in Y$, and again by Lemma \ref{lemma:principal} $\flat\leq\star$. This contradicts $\star\notin Y=Y^\uparrow$.

Hence, $\inf X\geq \inf Y$ implies $X\subseteq Y$. In particular, if $\inf X=\inf Y$ then $X\subseteq Y$ and $Y\subseteq X$, and thus $\Psi$ is injective. Since, by the previous point, $\Psi$ is surjective, it is bijective. Moreover, this also shows that the inverse of $\Psi$ is order-preserving; hence, $\Psi$ is an order isomorphism.
\end{proof}

Note that the Noetherian hypothesis in part \ref{prop:xcal:bij} is necessary: for example, consider a valuation domain $V$ whose maximal ideal is branched. If $X:=\inssemistable(V)\setminus\{d_V,v_V\}$, and $Y:=\inssemistable(V)\setminus\{d_V\}$, then $\inf X=\inf Y=v_V$, but $X\neq Y$.

\begin{prop}\label{prop:noeth-max}
Let $R$ be a Pr\"ufer domain with Noetherian spectrum. Let $\inssmstable(R)$ be the set of stable (semi)star operations on $R$ and let $\mathcal{M}$ be the set of nondivisorial maximal ideals of $R$.
\begin{enumerate}[(a)]
\item\label{prop:noeth-max:max} For any $\star\in\inssmstable(R)$, there are disjoint subsets $\Delta_1,\Delta_2\subseteq\Max(R)$ such that $\Max(R)=\Delta_1\cup\Delta_2$ and such that
\begin{equation*}
I^\star=\bigcap_{P\in\Delta_1}IR_P\cap\bigcap_{P\in\Delta_2}(IR_P)^{v_{R_P}}
\end{equation*}
for every $I\in\inssubmod(R)$.
\item\label{prop:noeth-max:M} Endow the power set $\power(\mathcal{M})$ with the containment order. The map
\begin{equation*}
\begin{aligned}
\Phi\colon\inssmstable(R) & \longrightarrow \power(\mathcal{M})\\
\star & \longmapsto \psspec{\star}(R)
\end{aligned}
\end{equation*}
is an order isomorphism.
\item\label{prop:noeth-max:num} $|\inssmstable(R)|=2^{|\mathcal{M}|}$.
\end{enumerate}
\end{prop}
\begin{proof}
\ref{prop:noeth-max:max} We set $\Delta_1:=\qspec{\star}(R)\cap\Max(R)$ and $\Delta_2:=\psspec{\star}(R)$, and proceed to show that they fulfill our claim. We first claim that $\Delta_2=\Max(R)\setminus\Delta_1$. Indeed, by Corollary \ref{cor:noethspec}, we have
\begin{equation*}
R=R^\star=\bigcap_{P\in\qspec{\star}(R)}R_P\cap\bigcap_{P\in\psspec{\star}(R)}R_P= \bigcap_{P\in\qspec{\star}(R)\cup\psspec{\star}(R)}R_P.
\end{equation*}
However, since $\Spec(R)$ is Noetherian, $\qspec{\star}(R)\cup\psspec{\star}(R)$ must contain $\Max(R)$ (this is implied by \cite[Theorem 4.2.34]{fontana_libro}); since they are disjoint, the claim follows.

Hence, we can write, for every $I\in\inssubmod(R)$,
\begin{equation*}
I^\star=\bigcap_{P\in\Delta_1}IR_P\cap\bigcap_{P\in\Delta_2}(IR_P)^{v_{R_P}}\cap\bigcap_{P\in\Spec(R)\setminus\Max(R)}IR_P.
\end{equation*}
Since $\Spec(R)$ is Noetherian, every nonmaximal prime $P$ is divisorial \cite[Corollary 4.1.11]{fontana_libro} and thus (being $R=R^\star$) in $\qspec{\star}(R)$; moreover, if $P\subseteq M$, then $IR_P\subseteq IR_M$, so that the third big intersection of the above formula can be thrown out. The claim is proved.

\ref{prop:noeth-max:M} Let $\star\in\inssmstable(R)$: by the previous point, $\psspec{\star}(R)\subseteq\Max(R)$. Moreover, since $\star$ is a (semi)star operation, every divisorial ideal must be $\star$-closed, and in particular cannot belong to the pseudo-spectrum; hence, $\Phi$ is well-defined. Note also that, if $\star_1\leq\star_2$, then $\qspec{\star_1}(R)\supseteq\qspec{\star_2}(R)$; by the previous point, it follows that $\psspec{\star_1}(R)\subseteq\psspec{\star_2}(R)$. Therefore, $\Phi$ is order-preserving.

Let now $\Lambda\subseteq\mathcal{M}$. Define a semistar operation $\star_\Lambda$ by
\begin{equation*}
I\mapsto\bigcap_{P\in\Max(R)\setminus\Lambda}IR_P\cap\bigcap_{P\in\Lambda}(IR_P)^{v_{R_P}}.
\end{equation*}
Then, each $\star_\Lambda$ is a stable (semi)star operation on $R$, and by the previous part of the proof every stable (semi)star operation must be equal to $\star_\Lambda$ for some $\Lambda$; thus, the assignment $\Lambda\mapsto\star_\Lambda$ defines a surjective map $\Phi_0$ from $\power(\mathcal{M})$ to $\inssmstable(R)$. Moreover, $IR_P\subseteq(IR_P)^{v_{R_P}}$ for every $I$ and every prime $P$, and thus if $\Lambda_1\subseteq\Lambda_2$ then $\star_{\Lambda_1}\leq\star_{\Lambda_2}$. Hence, $\Phi_0$ is order-preserving.

To show that it is the inverse of $\Phi$, it is enough to show that $\psspec{\star_\Lambda}(R)=\Lambda$ for every $\Lambda\subseteq\mathcal{M}$. Clearly, if $P\in\Max(R)\setminus\Lambda$ then $P\in\qspec{\star_\Lambda}(R)$. On the other hand, let $P\in\Lambda$; since $\Lambda\subseteq\mathcal{M}$, the ideal $P$ is not divisorial over $R$, and we claim that $PR_P$ is not divisorial over $R_P$. If it is divisorial, then $PR_P=pR_P$ for some $p\in P$; moreover, since $\Spec(R)$ is Noetherian, there is a finitely generated ideal $I$ such that $V(I)=\{P\}$. If $J=(I,p)$, then $J$ is a $P$-primary ideal such that $JR_P=PR_P$; hence, $J=P$ is invertible, and $P$ is divisorial, a contradiction.

By definition, $P^{\star_\Lambda}=(PR_P)^{v_{R_P}}\cap R$. However, by the previous reasoning $(PR_P)^{v_{R_P}}=R_P$, and thus $P^{\star_\Lambda}=R$; by the previous point, it must be $P\in\psspec{\star_\Lambda}(R)$. Hence, $\Phi_0$ is the inverse of $\Phi$, and both are order isomorphisms.

\ref{prop:noeth-max:num} is an immediate consequence of \ref{prop:noeth-max:M}.
\end{proof}

\begin{oss}
\begin{enumerate}
\item In the previous proposition, ``(semi)star operation'' can be substituted with ``star operation'' without problems, since there is a one-to-one correspondence between stable star and stable (semi)star operations.

\item If we focus on semistar operations, the natural extension of the first part of the previous proposition would be to ask for the existence of $\Delta_1,\Delta_2\subseteq\Spec(R)$. However, this is essentially Corollary \ref{cor:MinIfin-glob}.

\item Since each stable semistar operation $\star$ on $R$ can be seen as a stable (semi)star operation on $R^\star$, the previous proposition shows that $\inssemistable(R)$ is the disjoint union of a family of sets order-isomorphic to power sets (namely, the $\inssmstable(D)$, as $D$ ranges among the overrings of $R$). It is not clear if it is possible to obtain a good description of the whole set $\inssemistable(R)$ from this point of view.
\end{enumerate}
\end{oss}

The following is a small extension of Corollary \ref{cor:noethspec}, allowing the possibility of a ``small'' deviation from the Noetherianity of $\Spec(R)$.
\begin{prop}
Let $R$ be a Pr\"ufer domain, and let $\star$ be a stable (semi)star operation on $R$. If $\Spec(R)\setminus\qspec{\star}(R)$ is a Noetherian space, then $\star=\norm{\star}$.
\end{prop}
\begin{proof}
Again, we have to prove that $1\in I^\star$ if and only if $1\in I^{\norm{\star}}$, and one implication follows from Theorem \ref{teor:bigleq}. Suppose $1\in I^{\norm{\star}}$; then, $I$ is not contained in any $P\in\qspec{\star}(R)$, and thus all the minimal prime ideals of $I$ are in $\Spec(R)\setminus\qspec{\star}(R)$. Since the latter space is Noetherian, $\Min(I)$ is finite, and thus we can apply Theorem \ref{teor:MinIfin}.
\end{proof}

Recall that a semistar operation is \emph{spectral} if is in the form $I\mapsto\bigcap\{IR_P\mid P\in\Delta\}$ for some $\Delta\subseteq\Spec(R)$.
\begin{cor}
Let $R$ be a Pr\"ufer domain with Noetherian spectrum such that $PR_P$ is principal for every $P\in\Spec(R)$. Then, every stable semistar operation is spectral.
\end{cor}
\begin{proof}
By Corollary \ref{cor:noethspec}, every stable semistar operation is equal to its normalized stable version. By Lemma \ref{lemma:princ-ps}, $\psspec{\star}(R)=\emptyset$; hence, $\star$ can be written as $I\mapsto\bigcap\{IR_P\mid P\in\qspec{\star}(R)\}$, and thus $\star$ is spectral.
\end{proof}

\section*{Acknowledgements}
The author wishes to thank the referee for his/her suggestions, which greatly improved the paper.

\end{document}